\documentclass[11pt,titlepage]{article}
\usepackage{latexsym}
\usepackage{amsfonts}
\usepackage{amsmath}
\usepackage{amsthm}
\usepackage[boxed]{algorithm2e}
\newcommand\blackslug{\hbox{\hskip 1pt \vrule width 4pt height 8pt depth 1.5pt
        \hskip 1pt}}
\newcommand\bbox{\hfill \quad \blackslug \medbreak}

\def\c{\hbox{-}\cdots\hbox{-}}

\newtheorem{theorem}{}[section]

\newcommand{\Proof}{\noindent{\bf Proof.}\ \ }

\title{$P_{k}$-freeness implies small dichromatic number}
\author{
Krzysztof Choromanski
Google Research\\
New York, NY, USA
}

\newtheorem{conjecture}{}[section]
\newtheorem{lemma}{}[section]

\newcommand{\Keywords}{$P_{k}$-free tournaments, acyclic colorings, transitive subsets, the Erd\H{o}s-Hajnal Conjecture}

\begin{document}
\maketitle
\begin{abstract}

We propose a purely combinatorial quadratic time algorithm that for any $n$-vertex $P_{k}$-free tournament $T$, where $P_{k}$ is a directed path of length $k$, finds in $T$ a
transitive subset of order $n^{\frac{c}{k\log(k)^{2}}}$. As a byproduct of our method, we obtain subcubic $O(n^{1-\frac{c}{k\log(k)^{2}}})$-approximation algorithm 
for the optimal acyclic coloring problem on $P_{k}$-free tournaments. Our results are tight up to the $\log(k)$-factor in the following sense:
there exist infinite families of $P_{k}$-free tournaments with largest transitive subsets of order at most $n^{\frac{c\log(k)}{k}}$.
As a corollary, we give tight asymptotic results regarding the so-called \textit{Erd\H{o}s-Hajnal coefficients} of directed paths.
These are some of the first asymptotic results on these coefficients for infinite families of prime graphs.

\end{abstract}

\maketitle {\bf Keywords:} \Keywords

\section{Introduction}
Graph coloring problem is of fundamental importance in computer science. 
In the undirected setting the task is to color all the vertices of the graph to use as few colors
as possible and in such a way that every color class induces an independent set. 
The \textit{chromatic number $\chi(G)$ of the undirected graph $G$} is the minimum number of colors that can be used under these constraints.
In the directed setting (\cite{lara}) the coloring needs to be done in such a way that every color class induces an acyclic digraph.
Such a coloring is called an \textit{acyclic coloring}.
In particular, when a graph to color is a tournament then each color class is a transitive subset (transitive subsets correspond in the tournament setting to the independent sets in the undirected one).
The number of colors in the optimal acyclic coloring of a digraph $D$ is called the \textit{dichromatic number $\chi_{a}(D)$ of the digraph $D$}.
Digraph colorings arise in several applications and are thoroughly used in kernel theory and tournament theory thus they attracted attention of many researchers.

The coloring problem is NP-hard but even the stronger statement is true: for every $\epsilon>0$ finding a $n^{1-\epsilon}$-approximation 
of the optimal coloring is NP-hard. Due to its hardness, many research efforts focused on finding good-quality colorings for several special
classes of graphs. For instance, a scenario when a graph $G$ under consideration is $k$-colorable for some $k>0$ was analyzed.
The best known coloring algorithms for that case give $n^{1-\frac{c}{k}}$-approximation, where $n=|G|$. The best constants $c$ are obtained
with the use of semidefinite programming (\cite{blum,karger,arora}). 
Another important special class of graphs to consider in the coloring context are graphs defined by forbidden patterns.
These appear in many places in graph theory. For instance, every graph with the topological ordering of vertices can be 
equivalently described as not having directed cycles and every transitive tournament - as not having directed triangles.
A finite graph is planar if and only if it does not contain $K_{5}$ (the complete graph on five vertices) or $K_{3,3}$ (complete bipartite graph on six vertices with
two equal-length color classes) as a minor. One of the deepest results in graph theory, the Robertson-Seymour theorem (\cite{robertson}), states that every family of graphs 
(not necessarily planar graphs) that is closed under minors can be defined by a finite set of forbidden minors. These classes include: forests, pseudoforests, linear forests
(disjoint unions of path graphs), planar and outerplanar graphs, apex graphs, toroidal graphs, graphs that can be embedded on the two-dimensional manifold, graphs with
bounded treewidth, pathwidth or branchwidth and many more. We should notice that not having a certain
graph as a minor is a much more restrictive assumption than not having a certain graph $H$ as an induced subgraph.
Other examples include classes of graphs that can be colored with significantly fewer than $\Omega(\frac{n}{\log(n)})$ colors.
For instance, $k$-colorable graphs mentioned before do not have as induced subgraphs these graphs $H$ that have largest independent sets of size smaller than $\frac{|H|}{k}$.  
Thus all those classes can be described as not having some forbidden structures (either induced
subgraphs in the undirected scenario or subdigraphs in the directed setting).

One of these classes of graphs is of particular interest. Those are $P_{k}$-free graphs, 
where $P_{k}$ is an undirected path of $k$ vertices.
Not much is known for structural properties of $P_{k}$-free graphs for $k \geq 5$.
In particular, it is an open question whether finding the largest independent set 
is NP-hard if the class is defined by a forbidden path $P_{k}$ and $k>5$. Coloring $P_{k}$-free graphs for $k \geq 5$ is known to be NP-hard. 
Similarly, no nontrivial approximation algorithms for coloring
$P_{k}$-free graphs are known for $k>5$. Thus the question whether there exists a $n^{1-\frac{c}{k}}$-approximation algorithm
(as it is the case for $k$-colorable graphs) is open.
The completely analogous problem can be considered in the directed setting. In other words, one can ask for an optimal acyclic coloring
of $P_{k}$-free tournaments, where this time $P_{k}$ stands for the directed path tournament, i.e. a tournament with the ordering of vertices
$(v_{1},...,v_{k})$ under which the backward edges are exactly of the form $(v_{i+1},v_{i})$.
Like in the undirected case, the structural theorem of $P_{k}$-free directed graphs is not known. In particular, the question whether the
acyclic coloring problem is NP-hard for this class of graphs is open.
It is striking though that the $O(n^{1-\frac{c}{k\log(k)^{2}}})$-approximation algorithm exists and this is one of our main results in this paper.
In fact we show a stronger result. We give an algorithm that constructs in the $P_{k}$-free tournament a transitive set of order $n^{\frac{c}{k\log(k)^{2}}}$.
We show that our results are tight up to the $\log(k)$-factor in the following sense:
there exist infinite families of $P_{k}$-free tournaments with largest transitive subsets of order at most $n^{\frac{c\log(k)}{k}}$. As a corollary, 
we give tight asymptotic results regarding the so-called \textit{Erd\H{o}s-Hajnal coefficients} for directed paths. The coefficients come from the
celebrated Erd\H{o}s-Hajnal conjecture - one of the most fundamental unsolved problems in modern Ramsey graph theory.
Our algorithm for finding big transitive subsets is quadratic in the size of the input thus optimal (since the input as a tournament is of size $\Theta(n^{2})$)
and easy to implement.
It leads straightforwardly to the subcubic coloring algorithm.

\section{Related work}

Let us discuss briefly some known results regarding $P_{k}$-free graphs.
Graph coloring problem is known to be solvable in the polynomial time for $P_{k}$-free graphs, where $k \leq 4$ (\cite{chvatal}).
We already mentioned that it was proven to be NP-hard for $k \geq 5$ (\cite{tuza}).
A related problem whether a given $P_{t}$-free graph is $k$-colorable (and finding the coloring if the $k$-coloring exists) 
was considered in several papers. In \cite{sgall} it was proven that the $3$-colorability question for $P_{5}$-free graphs can be answered
in the polynomial time. In fact $3$-coloring question can be answered in the polynomial time for a more general class of $P_{6}$-free graphs
(\cite{randerath}). A polynomial algorithm answering a question whether a $P_{5}$-free graph can be $k$-colored (and constructing the coloring 
if this is the case) for arbitrary $k>0$ was given in \cite{kaminski}.
Very recently a polynomial algorithm for constructing maximum independent set in $P_{5}$-free graphs was proposed (\cite{villanger}).  
No nontrivial approximation algorithms for the coloring problem of $P_{k}$-free graphs for general $k$ were proposed.

In the directed setting it was recently proven (\cite{choromanski}) that for every $k>0$ all $P_{k}$-free tournaments have polynomial-size transitive subsets, 
i.e transitive subsets of size $\Omega(n^{\epsilon})$ for some $\epsilon>0$. Coefficients $\epsilon$ were however obtained with the use of the regularity lemma,
an inherent ingredient of the entire approach, thus applied methods did not lead to any practically interesting algorithmic results. 
For paths and in fact all prime tournaments those coefficients were proven to be of order at most $\frac{\log(k)}{k}$ (\cite{choromanski2}) in the worst-case scenarios.
That led to the substantial gap between best known upper and lower bounds (the latter being only inversely proportional to the tower function from the regularity lemma).
We practically get rid of that gap in this paper.

Other results regarding (pseudo)transitive subtournaments of polynomial sizes in $H$-free directed graphs can be found in: \cite{strong}, \cite{allt}, \cite{chorojeb},
\cite{seymour}, \cite{pseudo}, \cite{pairs}, \cite{upper}, \cite{diss} and \cite{rank}.

\section{Main results}

Before stating formally our results, we will introduce a few important definitions used throughout this article.

All graphs in this paper are finite and simple. Let $G$ be a graph. The vertex 
set of $G$ is denoted by $V(G)$, and the edge set by $E(G)$. We write 
$|G|$ to mean $|V(G)|$. We
refer to $|G|$ as the {\em order} of $G$. A {\em clique} in an undirected graph $G$ is a subset of $V(G)$ all of whose elements are pairwise
adjacent, and an {\em independent set} in $G$ is a subset of $V(G)$ all of whose 
elements are pairwise non-adjacent. 
We say that a graph $G$ is $H$-free if it does not have $H$ as an induced subgraph.
A {\em tournament} is a directed graph $T$, where 
for every two vertices $u,v$ exactly one of $(u,v)$, $(v,u)$ is an edge of $T$ (that is, a directed edge). 
 If $(u,v) \in E(T)$, we say
that $u$ is {\em adjacent to} $v$, and that $v$ is {\em adjacent from} $u$.
Equivalently, $v$ is an \textit{outneighbor} of $u$ and $u$ is an \textit{inneighbor} of $v$.
For a given $X \subseteq V(T)$ we denote by $T|X$ a subtournament of $T$ induced by a vertex set $X$.
For a graph $G$ and a subset $V \in V(G)$
we denote by $G \setminus V$ a graph obtained from $G$ by deleting $V$ and all
edges of $G$ that are: adjacent to a vertex $v \in V$ in the undirected setting and: adjacent to or from
a vertex $v \in V$ in the directed setting.
A tournament is {\em transitive} if it contains no directed cycle (equivalently, no directed
cycle of length three).
A set is \textit{transitive} if it induces a transitive subtournament.
A \textit{homogeneous set} in a graph $G$ is a subset $V \subseteq V(G)$ such that if a vertex $v \in V(G) \setminus V$
is adjacent to a vertex of $V$ then it is adjacent to all the vertices of $V$.
A graph $G$ is \textit{prime} if all its homogeneous sets other than $V(G)$ are singletons.
For two disjoint subsets of the vertices $X,Y \subseteq V(G)$ of a graph $G$ we say that $\textit{$X$ is complete to $Y$}$ if every vertex of $X$
is adjacent to every vertex of $Y$.
(Last three definitions are valid in both undirected and directed setting).

A \textit{directed path $P_{k}$} (or simply \textit{a path $P_{k}$} if it is clear from the context that a graph under consideration is a tournament)
is a tournament with vertex set $V(P_{k}) = \{v_{1},...,v_{k}\}$ and an ordering of vertices $(v_{1},...,v_{k})$ under which the backward edges are 
exactly of the form: $(v_{i+1},v_{i})$ for $i=1,...,k-1$. We call this ordering a \textit{path ordering}. If $(v_{1},...,v_{k})$ is a path ordering of $P_{k}$, then
we call an ordering $(v_{1},v_{3},v_{2},v_{5},v_{4},...,)$ a \textit{matching ordering} since under this ordering a graph of backward edges is a matching.
Let $\mathcal{B}_{P_{k}}=(E_{1},...E_{\lfloor \frac{k}{2} \rfloor})$ be a sequence of backward edges of this ordering, where the backward edges are ordered in $\mathcal{B}_{P_{k}}$ according 
to the location of their left ends in the matching ordering of $P_{k}$.  For the $ith$ backward edge $E_{i}$ we denote by $left(i)$ the location in the matching ordering of $P_{k}$
of the left end of $E_{i}$ and by $right(i)$ the location in the matching ordering of $P_{k}$ of the right end of $E_{i}$ ($left(i),right(i) \in \{1,...,k\}$).
Notice that if $k \neq 4$ then a directed path $P_{k}$ is prime.

We are ready to state our results. Our main result is as follows.

\begin{theorem}
\label{alg-theorem1}
There exists a universal constant $c>0$ such that for any $k>0$ there is an algorithm finding a transitive 
set of order $n^{\frac{c}{k\log(k)^{2}}}$ in the $P_{k}$-free $n$-vertex  tournament in $O(n^{2})$ time. 
\end{theorem}

As a simple corollary we obtain:

\begin{theorem}
\label{alg-theorem2}
There exists a universal constant $c>0$ such that for any $k>0$ there is an algorithm constructing acyclic coloring
of the $P_{k}$-free $n$-vertex tournament using only $n^{1-\frac{c}{k\log(k)^{2}}}$ colors. Furthermore, the algorithm
has running time $O(n^{3-\frac{c}{k\log(k)^{2}}})$.
\end{theorem}

This result immediately implies the following:

\begin{theorem}
\label{thm-dichromatic}
There exists a universal constant $c>0$ such that the dichromatic number of the $P_{k}$-free $n$-vertex tournament satisfies:
$$\chi_{a}(P_{k}) \leq n^{1-\frac{c}{k \log(k)^{2}}}.$$
\end{theorem}

It also serves as the $O(n^{1-\frac{c}{k\log(k)^{2}}})$-approximation algorithm for the optimal acyclic coloring of the $P_{k}$-free $n$-vertex graph.

Let us switch now to the conjecture of Erd\H{o}s and Hajnal.
The conjecture (\cite{erdos0}) says that:

\begin{conjecture}
\label{erdos-hajnal-conjecture-undirected}
For every undirected graph $H$ there exists a constant $\epsilon(H)>0$ such that the following holds: every $H$-free graph $G$ contains a clique or a stable set
of size at least $|G|^{\epsilon(H)}$.
\end{conjecture}

In its directed equivalent version (\cite{alon}) undirected graphs are replaced by tournaments and cliques/stable sets by transitive subtournaments:

\begin{conjecture}
\label{erdos-hajnal-conjecture-directed}
For every tournament $H$ there exists a constant $\epsilon(H)>0$ such that the following holds: every $H$-free tournament $T$ contains a transitive subtournament of order at least $|T|^{\epsilon(H)}$.
\end{conjecture}

The coefficient $\epsilon(H)$ from the statement of the conjecture is called \textit{the Erd\H{o}s-Hajnal coefficient}.
The conjecture was proven so far only for some very special forbidden patterns. Those of them that are prime are particularly important since if the conjecture is true for all prime graphs
then it is true in general (\cite{alon}).
There are no prime undirected graphs of order at least six for which the conjecture is known to be true and for a long time that was the case also in the directed setting.
Very recently an infinite family of prime tournaments satisfying the conjecture was constructed (\cite{choromanski}). Among them were directed paths $P_{k}$.
The proof of the conjecture for them provided only purely theoretical guarantees since all lower bounds for $\epsilon(H)$ were obtained by the regularity lemma.
Our algorithm gives lower bounds on the Erd\H{o}s-Hajnal coefficient that are very close to the best upper bounds since we have the following (\cite{choromanski2}):

\begin{theorem}
\label{uppertheorem}
There exists a constant $c>0$ such that every prime tournament $H$ satisfies:
$$\epsilon(H) \leq \frac{c\log(|H|)}{|H|}.$$
\end{theorem}

Combining this result with the lower bounds produced by our algorithm, we obtain the following result regarding the asymptotic behaviour of the Erd\H{o}s-Hajnal coefficients
of directed paths $P_{k}$:

\begin{theorem}
\label{simple-theorem}
The Erd\H{o}s-Hajnal coefficient of the directed path $P_{k}$ satisfies:
$$\epsilon(P_{k}) = \frac{1}{k^{1+o(1)}}.$$
\end{theorem}

So far such precise asymptotics were known only for one infinite family of prime tournaments, so-called \textit{stars} (see: \cite{choromanski} for a definition of a star).
Our results make the family of directed paths the second class of prime tournaments for which these asymptotics are known.

In the next section we present algorithms mentioned in Theorem \ref{alg-theorem1} and Theorem \ref{alg-theorem2}.
In the following section we prove that both algorithms have properties described in these theorems.
In the last section we summarize our results and briefly discuss possible extensions of the presented techniques.

\section{The Algorithm}

All considered logarithms are of base two from now on.  Without loss of generality we will assume that $k=2^{w}$ for some $w>0$.
First we will present an algorithm \textit{FindTrans} that finds in the $P_{k}$-free $n$-vertex tournament a transitive subset of size $n^{\frac{c}{k\log(k)^{2}}}$ for some universal
constant $c>0$ (an exact value of this constant may be calculated, but we will not focus on it in the paper).
The acyclic coloring algorithm \textit{AcyclicColoring} is a simple application of the former. It runs \textit{FindTrans} to find the first color class, removes it from the tournament,
runs \textit{FindTrans} on the remaining tournament to find the second color class, etc.

\subsection{Algorithm  \textit{FindTrans}}

Before giving a description of the algorithm \textit{FindTrans}, we need to introduce a few more definition. 
For a tournament $T$ and two disjoint nonempty subsets $X,Y \subseteq V(T)$ we denote $d(X,Y) = \frac{e(X,Y)}{|X||Y|}$,
where $e(X,Y)$ is the number of directed edges of $T$ going from $X$ to $Y$. The expression $d(X,Y)$ basically encodes directed
density of edges from $X$ to $Y$.

\begin{algorithm}
\SetAlgoLined
\textbf{Input: } $k>1$ and $\alpha$-sequence $\theta=(A_{1},...,A_{k})$ of length $k$\;
\textbf{Output: } $\alpha$-sequence $\theta_{s}$\;
\Begin{
let $\lambda = \frac{1}{32k^{4}}$ and $\lambda_{k} = 4\lambda k^{2}$\;
let $C_{i,j}=\{v \in A_{i} : |N^{\theta}_{v}(j)| \leq |A_{j}|(1-2k \lambda_{k})\}$ for $i,j \in \{1,...,h\}$, $i \neq j$\;
update: $A_{i} \leftarrow A_{i} \setminus \bigcup_{j \neq i} C_{i,j}$\;
output $(A_{1},...,A_{k})$\;
}
\caption{Algorithm \textit{MakeSmooth}}
\label{makesmooth}
\end{algorithm}

We say that a sequence $(A_{1},...,A_{l})$
of pairwise disjoint subsets of $V(T)$ is a \textit{$(c,\lambda)$-$\alpha$-sequence of length $l$} if the following holds: 
\begin{itemize}
\item $|A_{i}| \geq c|T|$ for $i=1,...,l$ and
\item $d(A_{i},A_{j}) \geq 1 - \lambda$ for $1 \leq i < j \leq l$.
\end{itemize}

If the parameters $c$, $\lambda$ of the \textit{$(c,\lambda)$-$\alpha$-sequence} are not important, we simply refer to it as an
$\alpha$-sequence.

We say that a $(c,\lambda)$-$\alpha$-sequence $(A_{1},...,A_{l})$ of length $l$ is \textit{smooth} if
the following strenghtening of the second condition from the definition above holds: 
\begin{itemize}
\item $d(\{x\},A_{j}) \geq 1 - \lambda$ for $x \in A_{i}$, $1 \leq i < j \leq l$ and,
\item $d(A_{i},\{y\}) \geq 1 - \lambda$ for $y \in A_{j}$, $1 \leq i < j \leq l$.
\end{itemize} 

Given an $\alpha$-sequence $\theta = (A_{1},...,A_{l})$, a vertex $v \in A_{i}$ and $j \neq i$ we denote by $N^{\theta}_{v}(j)$:
\begin{itemize}
\item a set of all outneighbors of $v$ from $A_{j}$ if $j>i$ and,
\item a set of all inneighbors of $v$ from $A_{j}$ if $j<i$. 
\end{itemize}

For an $\alpha$-sequence $\theta=(A_{1},...,A_{l})$ we denote: $V(\theta)=A_{1} \cup ... \cup A_{l}$.
Let $\theta_{1}=(A_{1},...,A_{l})$ and $\theta_{2}=(B_{1},...,B_{r})$ be two disjoint $\alpha$-sequences.
We denote by $\theta_{1} \otimes \theta_{2}$ the $\alpha$-sequence $(A_{1},...,A_{l},B_{1},...,B_{r})$.
For a set $A$ and $m \leq |A|$ we denote by $tr(A, m)$ the truncated version of $A$ obtained by taking arbitrarily its
$m$ elements.  
For an $\alpha$-sequence $\theta = (A_{1},...,A_{l})$ we denote: $tr(\theta, m)=(tr(A_{1},m),...,tr(A_{l},m))$.

If the order of the given $P_{k}$-free tournament is too small, the algorithm \textit{FindTrans} (Algorithm \ref{findtrans}) returns a trivial answer (and it is easy to see
that this gives good asymptotics on the coefficient $\epsilon$). 

\begin{algorithm}
\SetAlgoLined
\textbf{Input: } $k>0$ and $P_{k}$-free tournament $T$\;
\textbf{Output: } transitive subset in $V(T)$ of order $|T|^{\frac{c}{k\log(k)^{2}}}$\;
\Begin{
\If {$|T|=1$} {
output $V(T)$\;
}
let $c_{k} = \frac{1}{k} (\frac{\lambda^{k}}{k^{2}})^{\log(k)+1}$, where: $\lambda=\frac{1}{32k^{4}}$\;
\If{$1 < |T| \leq \frac{k}{c_{k}}$} {
output any $2$-element subset of $V(T)$\; 
}
run \textit{CreateSequence(k,T)} to obtain an $\alpha$-sequence $\theta$ of length $k$\;
run \textit{MakeSmooth(k, $\theta$)} to obtain a smooth $\alpha$-sequence $(A_{1},...,A_{k})$\;
initialize: $\theta_{s} \leftarrow (A_{1},...,A_{k})$\;
let $\theta_{s}(i)$ denote the $ith$ element of $\theta_{s}$ \;
\For{$i=1,...,\frac{k}{2}$} {
   let $u=left(i)$\;
   let $v=right(i)$\;
   \eIf {there exists an edge $e=(y,x)$ from $\theta_{s}(v)$ to $\theta_{s}(u)$}
   {
     let $A^{'}_{v} \leftarrow \theta_{s}(v)$,
         $A^{'}_{u} \leftarrow \theta_{s}(u)$ and\\
         $A^{'}_{t} \leftarrow \theta_{s}(t) \cap N^{\theta_{s}}_{y}(t) \cap N^{\theta_{s}}_{x}(t)$ for $t \in \{1,...,k\} \setminus \{v,u\}$\;
    update: $\theta_{s} \leftarrow (A^{'}_{1},...,A^{'}_{k})$\;
   }{
       run $FindTrans(k,T|A_{u})$ to obtain a transitive subset $M_{1}$\;
       run $FindTrans(k,T|A_{v})$ to obtain a transitive subset $M_{2}$\;
       output $M_{1} \cup M_{2}$\;
   }
}
}
\caption{Algorithm \textit{FindTrans}}
\label{findtrans}
\end{algorithm}

Otherwise, the algorithm uses two subprocedures: \textit{CreateSequence} that constructs in the $P_{k}$-free tournament $T$ an
$\alpha$-sequence of length $k$ and \textit{MakeSmooth} that uses that sequence to construct a smooth $\alpha$-sequence
of the same length. The $\textit{CreateSequence}$ procedure does not rely on the structural properties of the $P_{k}$-free tournaments,
but just uses the fact that a tournament it operates on is $H$-free for some $k$-vertex forbidden pattern $H$.
We will discuss it in much detail later. The \textit{MakeSmooth} procedure (Algorithm \ref{makesmooth}) is a standard method for getting rid of these vertices from the
given $\alpha$-sequence that have many less in/out-neighbors in some element of the $\alpha$-sequence than the density condition would suggest.

\begin{algorithm}
\SetAlgoLined
\textbf{Input: } $r>1$ and $P_{k}$-free $n$-vertex tournament $T$\;
\textbf{Output: } an $\alpha$-sequence of length $r$ in $T$\;
\Begin{
let $V(P_{k})=\{h_{1},...,h_{k}\}$\;
partition $V(T)$ arbitrarily into $k$ equal-size sets: $S_{1},...,S_{k}$\;
run \textit{$MakeDensePair(\{S_{1},...,S_{k}\},P_{k}, n)$} to get $(X,Y)$, where $X,Y \subseteq V(T)$\;
\If {$r=2$} {
  output $(X,Y)$\;
}
initialize: $\mathcal{L} \leftarrow \emptyset$, $\mathcal{R} \leftarrow \emptyset$\;
let $s_{1}=|X|$ and $s_{2}=|Y|$\;
\While{$|X| \geq \frac{s_{1}}{2}$}{
  run $CreateSequence(\frac{r}{2},T|X)$ to obtain an $\alpha$-sequence $L$ of length $\frac{r}{2}$\;
  update: $X \leftarrow X \setminus V(L)$, $\mathcal{L} \leftarrow \mathcal{L} \cup \{L\}$\;
}
\While{$|Y| \geq \frac{s_{2}}{2}$}{
  run $CreateSequence(\frac{r}{2},T|Y)$ to obtain an $\alpha$-sequence $R$ of length $\frac{r}{2}$\;
  update: $Y \leftarrow Y \setminus V(R)$, $\mathcal{R} \leftarrow \mathcal{R} \cup \{R\}$\;
}
let $\lambda = \frac{1}{32k^{4}}$\, $c_{r} = (\frac{c}{2})^{\log(r)-1}c$ and $m=\frac{cn}{2} c_{\frac{r}{2}}$, where $c=\frac{\lambda^{k}}{k^{2}}$\;
\eIf{exists $L \in \mathcal{L}$, $R \in \mathcal{R}$ such that $d(V(L),V(R)) \geq 1 - 4\lambda$}{
 output $tr(L \otimes R, m)$\;
}{
 output $\theta_{1} \otimes \theta_{2}$, for arbitrary: $\theta_{1} \in \mathcal{L}$ and $\theta_{2} \in \mathcal{R}$\;
}
}
\caption{Algorithm \textit{CreateSequence}}
\label{createsequence}
\end{algorithm}

Let us assume now that a smooth $\alpha$-sequence of length $k$ is given.
The algorithm \textit{FindTrans} tries to reconstruct a directed path $P_{k}$ by looking for its $ith$ vertex in the matching ordering of $P_{k}$ in the
$ith$ element of the $\alpha$-sequence $\theta_{s}$. This is conducted backward edge by backward edge in the matching ordering of $P_{k}$. If the backward edge 
is not found then two linear-size subsets $A_{u},A_{v}$ of two distinct elements from the original $\alpha$-sequence such that $A_{u}$ is complete to $A_{v}$ are detected. 
Otherwise an $\alpha$-sequence is updated.
The update is done in such a way that if in the new $\alpha$-sequence the other backward edges of the matching ordering of $P_{k}$ are found then
they can be combined with the edges that were already found to reconstruct a copy of $P_{k}$.
Since a tournament $T$ the algorithm is working on is $P_{k}$-free, at some point of its execution two subsets $A_{u},A_{v}$ mentioned above with $d(A_{u},A_{v})=1$ will be detected.
When that happens, the algorithm is run recursively on the tournaments: $T|X$ and $T|Y$ and later two transitive subsets found in these two recursive runs are merged.

Let us discuss now subprocedure \textit{CreateSequence} (Algorithm \ref{createsequence}) that constructs an $\alpha$-sequence of a specified length $r$
(without loss of generality we will assume that $r=2^{w}$ for some $w>0$). 
As mentioned before, the procedure can be applied for any forbidden pattern, not only $P_{k}$. 
Its main ingredient is called \textit{MakeDensePair} and is responsible for constructing two disjoint
linear sets $X,Y$ in the $P_{k}$-free tournament such that the directed density $d(X,Y)$ is very close to one.

\begin{algorithm}
\SetAlgoLined
\textbf{Input: } a set $\{S_{i_{1}},...,S_{i_{p}}\}$ such that $S_{i_{1}} \cup ... \cup S_{i_{p}}$ induces a $P_{k}$-free tournament, a $p$-vertex tournament $H$ with $V(H)=\{h_{i_{1}},..,h_{i_{p}}\}$
                       and parameter $n$\;
\textbf{Output: } a pair of disjoint sets $(X,Y)$\;
\Begin{
let $\lambda = \frac{1}{32k^{4}}$ and $m=\frac{\lambda^{k}}{k^{2}}n$\;
for each $v \in S_{i_{1}}$ and $j \in \{i_{2},...,i_{p}\}$ let $N(v,S_{j})$ be: \\
$\empty$ $\empty$ $\empty$ a set of outneighbors of $v$ from $S_{j}$ if $(h_{i_{1}},h_{j})$ is an edge and:\\ 
$\empty$ $\empty$ $\empty$ a set of inneighbors of $v$ from $S_{j}$ otherwise\;
let $bad(v)$ be: an arbitrary  $j \in \{i_{2},...,i_{p}\}$ such that $|N(v,S_{j})| < \lambda |S_{j}|$ or \\ 
$\empty$ $\empty$ $\empty$ $\empty$ $\empty$ $\empty$ $\empty$ $\empty$ $\empty$ $\empty$ $\empty$ $\empty$ $\empty$ $\empty$ 
$\empty$ $\empty$ $\empty$ $\empty$ $\empty$ $0$ if such a $j$ does not exist\;
\eIf {there exists $v_{0} \in S_{i_{1}}$ such that $bad(v_{0})=0$}
{
update: $S_{j} \leftarrow tr(N(v_{0},S_{j}), \lambda |S_{j}|)$ for $j \in \{i_{2},...,i_{p}\}$\;
let $\mathcal{S}_{new} \leftarrow \{S_{j}: j \in \{i_{2},...,i_{p}\}\}$\;
run MakeDensePair($\mathcal{S}_{new}$, $H \setminus \{h_{i_{1}}\}$, n)\;
}
{
  let $P_{j}=\{v \in S_{i_{1}}: bad(v)=j\}$ for $j \in \{i_{2},...,i_{p}\}$\;
  let $j_{0} =  \arg \max_{j \in \{i_{2},...,i_{p}\}} |P_{j}|$ and $P_{j_{0}}^{t}=tr(P_{j_{0}},m)$\;
  let $\{W_{1},W_{2},...\}$ be a partitioning of $S_{j_{0}}$ into sets of size $m$\; 
  \eIf{$d(P_{j_{0}},S_{j_{0}}) \geq \frac{1}{2}$}
  {
    output $(P_{j_{0}}^{t},W_{l_{max}})$, where $l_{max} = \arg \max_{l} d(P_{j_{0}}^{t},W_{l})$\; 
  }
  {
    output $(W_{l_{max}},P_{j_{0}}^{t})$, where $l_{max} = \arg \max_{l} d(W_{l}, P_{j_{0}}^{t})$\;
  }
}
}
\caption{Algorithm \textit{MakeDensePair}}
\label{makedensepair}
\end{algorithm}

The procedure \textit{CreateSequence} acts as follows.
First two linear sets $X,Y$ of the $P_{k}$-free tournament and with $d(X,Y) \geq 1 - \lambda$ for some $0<\lambda \ll 1$ are found with the use of 
the procedure \textit{MakeDensePair}. If $r=2$ then $(X,Y)$ is output and the procedure is ends. Otherwise, in both $X$ and $Y$ the $\alpha$-sequences of length $\frac{r}{2}$ are constructed
recursively. When the sequence is constructed, it is deleted from $X$ or $Y$ and a new sequence is being constructed in the remaining set. This is repeated as long there are at least
half of the vertices left in $X$ or $Y$. Let $X_{1},X_{2},...$ denote the sets of the vertices of the $\alpha$-sequences constructed in $X$ and let 
$Y_{1},Y_{2},...$ denote the sets of the vertices of the $\alpha$-sequences constructed in $Y$.
The algorithm is looking for sets $X_{i}$,$Y_{j}$ such that $d(X_{i},Y_{j}) \geq 1 - 4\lambda$.
The way sets $X,Y$ were constructed by \textit{MakeDensePair} as well as simple density arguments (see: the analysis of the algorithm) 
imply that such sets do exist. Thus even though in the formal description of  \textit{CreateSequence} we assume that the sets may not be found 
(and then two arbitrary sets $X_{i}$, $Y_{j}$) are taken, this in fact will never happen.
The $\alpha$-sequence of length $r$ is output simply by combining two $\alpha$-sequences of length $\frac{r}{2}$ corresponding to $X_{i}$ and $Y_{j}$.

It remains to explain how the procedure \textit{MakeDensePair} works (Algorithm \ref{makedensepair}).
The procedure is given a set of sets $S_{i_{1}},...,S_{i_{p}} \subseteq V(T)$ of linear size each, for some $1 < p \leq k$, a $p$-vertex tournament $H=\{h_{i_{1}},...,h_{i_{p}}\}$,
and a parameter $n$. Parameter $n$ is the remembered size of the tournament which is an input of
the \textit{CreateSequence} procedure initializing the recursive runs of  \textit{MakeDensePair}.

Notice that $T|S_{i_{1}} \cup ... \cup S_{i_{p}}$ is $H$-free.
The procedure tries to reconstruct $H$ in $T|S_{i_{1}} \cup ... \cup S_{i_{p}}$ in such a way that $h_{i_{j}}$ is found in $S_{i_{j}}$.
It first verifies whether a good candidate for $h_{i_{1}}$ exists in $S_{i_{1}}$.
A good candidate should have substantial number of outneighbors in each $S_{i_{j}}$ such that $(h_{i_{1}},h_{i_{j}})$ is an edge in $H$ and
a substantial number of inneighbors in each $S_{i_{j}}$ such that $(h_{i_{j}},h_{i_{1}})$ is an edge in $H$.
If such a vertex $v$ in $S_{i_{1}}$ is found then the remaining sets are modified accordingly and the algorithm tries to reconstruct 
$H \setminus \{h_{i_{1}}\}$ in their modified versions. This is done by a recursive run of the procedure on the set of modified sets $S_{i_{2}},...,S_{i_{p}}$. 
Since a tournament that the procedure operates on is $H$-free, at some recursive run
no good candidate will be found. As we will see in the theoretical analysis, it will imply (by Pigeonhole Principle) the existence of two linear-size sets $X,Y$ with
density $d(X,Y)$ close to one. These sets will be output by the procedure. 

\begin{algorithm}
\SetAlgoLined
\textbf{Input: } $k>0$ and $P_{k}$-free tournament $T$\;
\textbf{Output: } an acyclic coloring of $T$ using $|T|^{1-\frac{c}{k\log(k)^{2}}}$ colors\;
\Begin{
initialize: $G \leftarrow T$, $\mathcal{P} \leftarrow \emptyset$\;
\While{$V(G) \neq \emptyset$}{
  run \textit{FindTrans(k, G)} to obtain a transitive set $M$ in $G$\;
  update: $\mathcal{P} \leftarrow \mathcal{P} \cup \{M\}$, $G \leftarrow G \setminus M$\;
}
color each set of $\mathcal{P}$ with different color and output this coloring\;
}
\caption{Algorithm \textit{AcyclicColoring}}
\label{alg-color}
\end{algorithm}

The use of parameter $n$ enables us to output two sets of the same desired
size. This balanceness will play important role in the theoretical analysis of the procedure \textit{CreateSequence} that uses \textit{MakeDensePair}.

\subsection{Algorithm  \textit{AcyclicColoring}}

The acyclic coloring algorithm (Algorithm \ref{alg-color}) is a simple wrapper for the \textit{FindTrans} procedure. It runs this procedure several times to obtain the partitioning of the $P_{k}$-
free tournament $T$ into transitive sets.
Each transitive set gets its own color and this coloring is as an acyclic coloring that is being output.

\section{Analysis}

\subsection{Introduction}

To show that presented algorithms are correct we need to prove Theorem \ref{alg-theorem1} and Theorem \ref{alg-theorem2}.
Let us assume first that Theorem \ref{alg-theorem1} is true. Under this assumption it is easy to prove Theorem \ref{alg-theorem2}.

\Proof
Let $\epsilon = \frac{c}{k\log(k)^{2}}$, where $k$ is as in Theorem \ref{alg-theorem1}.
The Algorithm \ref{alg-color} keeps finding transitive subtournaments of order at least $(\frac{n}{2})^{\epsilon}$ as long there are at least $\frac{n}{2}$ vertices left in the tournament. 
By the time the algorithm reaches the state with less than $\frac{n}{2}$ vertices remaining, at most $O(n^{1-\epsilon})$ transitive subtournaments are found.
Then the algorithm is run on the remaining graph of less than $\frac{n}{2}$ vertices. The algorithm stops when there are no vertices left. When it happens all the vertices of the tournament
are partitioned into transitive subsets.
If we denote by $H(n)$ the total number of the transitive subtournaments found then we have the following simple recurrence formula: $H(n) \leq O(n^{1-\epsilon}) + H(\frac{n}{2})$, which immediately gives us: 
$H(n) = O(n^{1-\epsilon})$. Thus we obtain desired approximation of the acyclic coloring problem.
Since finding each transitive subset takes quadratic time and at most $O(n^{1-\epsilon})$ transitive subsets are constructed, the total running time of the coloring algorithm is as stated in Theorem \ref{alg-theorem2}.
\bbox

Theorem \ref{alg-theorem1} is a result of the series of lemmas:

\begin{lemma}
\label{lemma1}
Let $\lambda = \frac{1}{32k^{4}}$.
A run of Algorithm \textit{MakeDensePair} from \textit{CreateSequence} outputs two disjoint subsets $X,Y$ of the given $n$-vertex tournament such that $d(X,Y) \geq 1 - \lambda$
and $|X|=|Y| = cn$ where: $c = \frac{\lambda^{k}}{k^{2}}$, provided that $n > k$. 
\end{lemma}

The next lemma gives us the parameters of the $\alpha$-sequence constructed by procedure \textit{CreateSequence}.

\begin{lemma}
\label{lemma2}
Let $\lambda = \frac{1}{32k^{4}}$.
If $n>\frac{k}{c_{r}}$ then Algorithm \textit{CreateSequence} constructs a $(c_{r},\lambda_{r})$-$\alpha$-sequence of length $r$ in the given $n$-vertex tournament,
where: $c_{r} = c \cdot (\frac{c}{2})^{\log(r)-1}$, $c=\frac{\lambda^{k}}{k^{2}}$, $\lambda_{r} = 4\lambda r^{2}$ for $r>2$
and $\lambda_{2}=\lambda$. Furthermore, each element of the constructed $\alpha$-sequence is of the same size $c_{r} n$.
\end{lemma}

The parameters of the smooth $\alpha$-sequence produced by \textit{MakeSmooth} from the input $\alpha$-sequence are given in the next lemma:

\begin{lemma}
\label{lemma3}
Let $\lambda = \frac{1}{32k^{4}}$.
Assume that the input to the Algorithm  $MakeSmooth$ is a $(c_{k},\lambda_{k})$-$\alpha$-sequence of length $k$ for some $c_{k}>0$ and $\lambda_{k}=4\lambda k^{2}$.
Then Algorithm $MakeSmooth$ from $\textit{FindTrans}$ procedure constructs a smooth $(\frac{c_{k}}{2},\lambda_{f})$-$\alpha$-sequence,
where: $\lambda_{f} = 4k \lambda_{k}$.
\end{lemma}

The proof of Theorem \ref{alg-theorem1} as well as the proofs of the above lemmas are given in the next subsection.

\subsection{Proof of Theorem \ref{alg-theorem1}}

We start with the following simple lemma.

\begin{lemma}
\label{densitylemma}
Let $T$ be a tournament.
Assume that for two disjoint subsets $X,Y \subseteq V(T)$ the following holds: $d(X,Y) \geq 1 - \lambda$ for some $\lambda < 1$. Assume that $X_{1} \subseteq X$, $Y_{1} \subseteq Y$, $X_{1} \geq c_{1} |X|$, $Y_{1} \geq c_{2} |Y|$ for some $0  < c_{1},c_{2} < 1$. Then $d(X_{1},Y_{1}) \geq 1 - \frac{\lambda}{c_{1}c_{2}}$.
\end{lemma}

\Proof
Let $e_{Y,X}$ be the number of directed edges from $Y$ to $X$ and let $e_{Y_{1},X_{1}}$ be the number of directed edges from $Y_{1}$ to $X_{1}$.
We have: $$e_{Y,X}=(1-d(X,Y))|X||Y| \leq \lambda|X||Y|,$$ since $d(X,Y) \geq 1 - \lambda$. Similarly: $e_{Y_{1},X_{1}}=(1-d(X_{1},Y_{1}))|X_{1}||Y_{1}|$.
Assume by contradiction that $d(X_{1},Y_{1}) < 1 - \frac{\lambda}{c_{1}c_{2}}$. Then, since  $X_{1} \geq c_{1} |X|$, $Y_{1} \geq c_{2} |Y|$, we have:  $e_{Y_{1},X_{1}} > \lambda |X||Y|$. Since $e_{Y,X} \geq e_{Y_{1},X_{1}}$, we get: $e_{Y,X} > \lambda |X||Y|$, contradiction.
\bbox

Let us assume that lemmas: \ref{lemma2} and \ref{lemma3} from the main body of the paper are correct.
We will first show how Theorem \ref{alg-theorem1} is implied by them. Then we will prove all three lemmas
(Lemma \ref{lemma1} will be used to prove Lemma \ref{lemma2}).
The proof of Theorem \ref{alg-theorem1} is given below. \\

\Proof

We will proceed by induction on the size of the $P_{k}$-free tournament $T$.
Let $\epsilon = \frac{C}{k\log(k)^{2}}$, where $C>0$ is a small enough universal constant.
Let us consider first the case when $|T| \leq \frac{k}{c_{k}}$, where $c_{k}$ is as in Algorithm \textit{FindTrans}.
In this setting $|T|$ is of the order $k^{C^{'} k\log(k)}$ for some universal constant $C^{'}>0$ so the output of the algorithm trivially satisfies conditions
of Theorem \ref{alg-theorem1}.
Now let us consider more interesting case when $|T| > \frac{k}{c_{k}}$. Notice then that the requirement from Lemma \ref{lemma2} regarding the size of the
input $n$-vertex $P_{k}$-free tournament is trivially satisfied.
Assuming that lemmas: \ref{lemma2} and \ref{lemma3} are true, we conclude that initially the $\alpha$-sequence $\theta_{s}$
from \textit{FindTrans} is a smooth $(\frac{c_{k}}{2}, \lambda_{f})$-$\alpha$-sequence, where: $c_{k}=c \cdot (\frac{c}{2})^{\log(k)-1}$,
$\lambda_{f} = 4k \lambda_{k}$, $\lambda_{k} = 4\lambda k^{2}$, $c=\frac{\lambda^{k}}{k^{2}}$ and $\lambda = \frac{1}{32k^{4}}$.
Now consider the for-loop in the algorithm. Notice that it cannot be the case that in each run of the loop an edge $e=(y,x)$ is found.
Indeed, assume otherwise and denote the set of edges found in all $\frac{k}{2}$ runs by $\{(y_{1},x_{1}),...,(y_{\frac{k}{2}},x_{\frac{k}{2}})\}$.
Denote by $\sigma(y_{i})$ this $j$ that satisfies: $y_{i} \in A_{j}$. Similarly, denote by $\sigma(x_{i})$ this $j$ that satisfies: $x_{i} \in A_{j}$.
Notice that the vertices $x_{1},y_{1},...,x_{\frac{k}{2}},y_{\frac{k}{2}}$ induce a copy of $P_{k}$ and besides the ordering of 
$\{x_{1},y_{1},...,x_{\frac{k}{2}},y_{\frac{k}{2}}\}$ induced by $\sigma$ is a matching ordering under which the set of backward edges is exactly:
 $\{(y_{1},x_{1}),...,(y_{\frac{k}{2}},x_{\frac{k}{2}})\}$. This is a straightforward conclusion from the way the $\alpha$-sequence $\theta_{s}$ is updated.
That however contradicts the fact that the tournament the algorithm operates on is $P_{k}$-free.
Thus we can assume that in some run of the main for-loop the algorithm recursively runs itself on $T|A_{u}$ and $T|A_{v}$ for some $A_{u},A_{v}$ from the given
$\alpha$-sequence. 
Notice that whenever a backward edge $(y,x)$ is found the size of each $A_{i}$ in the updated $\alpha$-sequence decreases by at most $2 \cdot \frac{c_{k}}{2}n \lambda_{f}$.
Thus at every stage of the execution of the algorithm each $A_{i}$ is of order at least $\frac{c_{k}}{2}n - k \cdot \frac{c_{k}}{2}n \lambda_{f}$ which is at least
$\frac{c_{k}}{4}n$ (since $\lambda_{f} = 4k \lambda_{k} \leq \frac{1}{2k}$).
Therefore when two recursive runs of the procedure \textit{FindTrans} are conducted, each run operates on the tournament of size at least $\frac{c_{k}}{4}n$.
By induction, a transitive tournament of order at least $2(\frac{c_{k}}{4}n)^{\epsilon}$ is produced.
It remains to prove that under our choice of $\epsilon$ (for $C>0$ small enough) we have: $2(\frac{c_{k}}{4}n)^{\epsilon} \geq n^{\epsilon}$, i.e
$\epsilon \leq \frac{1}{\log(\frac{4}{c_{k}})}$. We leave it to the reader.

Let us comment now on the running time of the algorithm. 
First notice that procedure \textit{MakeDensePair} runs in quadratic time. Throughout its execution it is calling itself at most $k$ times and the time it takes
between any two recursive calls is clearly at most quadratic. This in particular implies that procedure \textit{CreateSequence} also runs in quadratic time.
Indeed, throughout its execution at most $O(k)$ calls of the procedure \textit{MakeDensePair} are conducted and its other operations take altogether
at most quadratic time. 
Furthermore, Algorithm \textit{MakeSmooth} is clearly quadratic and besides a naive implementation of each run of the for-loop in the procedure \textit{FindTrans}
takes at most quadratic time. Thus Algorithm \textit{FindTrans} has quadratic running time.

\bbox

It remains to prove lemmas: \ref{lemma1}, \ref{lemma2}, \ref{lemma3}.
We start with Lemma \ref{lemma3}.

\Proof
Let $\theta=(A_{1},...,A_{k})$ be the input $\alpha$-sequence.
By Lemma \ref{densitylemma} we get: $|C_{i,j}| \leq \frac{|A_{i}|}{2k}$. Thus for any $i=1,...,k$ we have: $|\bigcup_{j \neq i} C_{i,j}| \leq \frac{|A_{i}|}{2}$.
This implies in particular that each updated $A_{i}$ is of size at least half the size of the original one.
Now take some $1 \leq i < j \leq k$ and a vertex $v \in A_{i}^{n}$, where $A_{i}^{n}$ is the new version of $A_{i}$ after the update.
By the definition of $A_{i}^{n}$ we know that $v$ has at most $2k \lambda_{k} |A_{j}|$ inneighbors from $A_{j}$.
Denote by $A_{j}^{n}$ the new version of $A_{j}$ after the update. Then we can conclude that $v$ has at most $4k \lambda_{k} |A_{j}^{n}|$
inneighbors from $A_{j}^{n}$. Similar analysis can be conducted for $1 \leq j < i \leq k$. That completes the proof.
\bbox

Now we prove Lemma \ref{lemma2} assuming that Lemma 1 is true.

\Proof
We proceed by induction on $r$. For $r=2$ Algorithm \textit{CreateSequence} is reduced to procedure \textit{MakeDensePair} thus  the result follows by
Lemma \ref{lemma1}. Let us assume now that $r>2$. 
Then, by induction and Lemma \ref{lemma1}, each element of each $\alpha$-sequence $L$ is of size at least $c_{\frac{r}{2}} \cdot \frac{cn}{2}$
and at most $c_{\frac{r}{2}} \cdot cn$.
Similarly, each element of each $\alpha$-sequence $R$ is of size at least $c_{\frac{r}{2}} \cdot \frac{cn}{2}$
and at most $c_{\frac{r}{2}} \cdot cn$.
In particular, the size of each element of an arbitrary $L \in \mathcal{L}$ is at most twice the size of each
element of an arbitrary $R \in \mathcal{R}$ and vice versa: the size of each element of an arbitrary $R \in \mathcal{R}$ is at most twice the size of each element of 
an arbitrary $R \in \mathcal{R}$.
By Lemma \ref{lemma1}, the directed density between initial sets $X$ and $Y$ is at least $1-\lambda$.
Denote $X_{1} = \bigcup_{L \in \mathcal{L}}$ and $Y_{1} = \bigcup_{R \in \mathcal{R}}$, where:
$\mathcal{L}$ and $\mathcal{R}$ are taken when both while-loops in the algorithm are completed. We trivially have:
$|X_{1}| \geq \frac{|X|}{2}$ and $|Y_{1}| \geq \frac{|Y|}{2}$. Thus by Lemma \ref{densitylemma}, we obtain:
$d(X_{1},Y_{1}) \geq  1 - 4\lambda$. 
Notice that $d(X_{1},Y_{1}) = \frac{\sum_{L \in \mathcal{L},R \in \mathcal{R}}d(V(L),V(R))|V(L)||V(R)|}{|X_{1}||Y_{1}|}$.
Let us assume first that there do not exist $L \in \mathcal{L}, R \in \mathcal{R}$ such that $d(V(L),V(R)) \geq 1-4\lambda$.
But then, by the above observation, we have: $d(X_{1},Y_{1}) <  \frac{\sum_{L \in \mathcal{L},R \in \mathcal{R}}(1-4\lambda)|V(L)||V(R)|}{|X_{1}||Y_{1}|}$.
Thus $d(X_{1},Y_{1}) < (1-4\lambda) \frac{\sum_{L \in \mathcal{L},R \in \mathcal{R}}|V(L)||V(R)|}{|X_{1}||Y_{1}|} = 1-4\lambda$, contradiction.
Therefore $\alpha$-sequences $L_{0},R_{0}$ such that $d(V(L_{0}),V(R_{0})) \geq 1-4\lambda$ will be found.
Notice that, by induction and Lemma \ref{lemma1} all elements of $L_{0}$ are of the same size. Similarly, all elements of $R_{0}$ are of the same size.
Thus, by our previous observations and Lemma \ref{densitylemma}, we can conclude that in the truncated version of the $R_{0}$-part of the output $\alpha$-sequence
the density between an element appearing earlier in the sequence and an element appearing later is at least $1-4 \lambda_{\frac{r}{2}}$.
Similarly, the directed density between an element of the final output that is from the $L_{0}$-part of the sequence and the one that is from
the $R_{0}$-part of the sequence is at least $1 - 4\lambda \cdot 4 (\frac{r}{2})^{2}$.
This leads us to the following recursive formula: $\lambda_{r} = \max(4 \lambda_{\frac{r}{2}}, 4\lambda \cdot 4 (\frac{r}{2})^{2})$ for $r>2$
and $\lambda_{2}=\lambda$. One can easily check that this recursion has a solution which is exactly of the form given in the statement of Lemma \ref{lemma2}.
Furthermore, trivially each element of the output $\alpha$-sequence is forced to be of order $\frac{cn}{2}c_{\frac{r}{2}}$, which leads to the recursive formula on
$c_{r}$ from the statement of Lemma \ref{lemma2}. 
\bbox 

It remains to prove Lemma \ref{lemma1}.

\Proof
Notice first that output sets $X$ and $Y$ are forced to be of the size given in the statement of Lemma \ref{lemma1}.
Indeed, sets: $P_{j_{0}}^{t}$ and $W_{i}$
are of size $m$ each which is exactly
$cn$ for $c=\frac{\lambda^{k}}{k^{2}}$.
The crucial observation is that the longest path in the tree of recursive calls of the procedure \textit{MakeDensePair} is of length at most
$k$. Assume otherwise and choose $k$ consecutive vertices $v_{0}$ constructed in $k$ consecutive recursive calls.
Denote these vertices as: $v_{0}^{1},...,v_{0}^{k}$. Notice that from the way each $v_{0}^{i}$ is constructed we can immediately deduce that
$\{v_{0}^{1},...,v_{0}^{k}\}$ induce a copy of $P_{k}$, contradiction.
So after the procedure \textit{MakeDensePair} is called first time by \textit{CreateSequence}, it executes at most $k$ its recursive calls.
Now notice that the size of the set $S_{i_{j}}$ from the input of the procedure decreases between its two consecutive recursive calls  exactly by 
a factor of $\frac{1}{\lambda}$. Thus when a set $P_{j_{0}}$ is found the size of $S_{j_{0}}$ is $\frac{n}{k}\lambda^{it}$, where $it\leq k$ is the number of recursive calls 
that were run.
By the definition of $P_{j_{0}}$ we have one of two possible options: 
\begin{itemize}
\item every vertex of $P_{j_{0}}$ is adjacent to at least $(1-\lambda)|S_{j_{0}}|$ vertices of $S_{j_{0}}$ or
\item every vertex of $P_{j_{0}}$ is adjacent from at least $(1-\lambda)|S_{j_{0}}|$ vertices of $S_{j_{0}}$.
\end{itemize}
\bbox
In particular we have: $d(P_{j_{0}}^{t},S_{j_{0}}) \geq 1 - \lambda$ or $d(P_{j_{0}}^{t},S_{j_{0}}) \leq \lambda$.
Assume without loss of generality that the former holds. Then by the same density argument as in the proof of Lemma \ref{lemma2} we 
can conclude that $d(P_{j_{0}}^{t},W_{l_{max}}) \geq 1 - \lambda$. 
Finally, notice that, as we have already mentioned at the very beginning of the proof, both $P_{j_{0}}^{t}$
and $W_{l_{max}}$ are of the desired length $m$. That completes the proof.

\subsection{Infinite families of $P_{k}$-free tournaments with small transitive subsets}

In this subsection we show that our results from the main body of the paper are tight up to the $\log(k)$-factor in the following sense: there exists an infinite family
of $P_{k}$-free tournaments with largest transitive subsets of order $O(n^{\frac{c\log(k)}{k}})$. Presented construction is based on
\cite{choromanski2}.
We need one more definition. Let $S, F$ be two tournaments and denote $V(S)=\{s_{1},...,s_{|S|}\}$. We denote by $S \times F$ a tournament $T$ with the vertex set
$V(T) = V_{1} \cup ... \cup V_{|S|}$, where each $V_{i}$ induces a copy of $F$ and for any $1 \leq i < j \leq |S|, x \in V_{i}, y \in V_{j}$ we have the following:
$x$ is adjacent to $y$ iff $s_{i}$ is adjacent to $s_{j}$ in $S$.

Fix $k>0$. Without loss of generality we can assume that $k>4$.
Notice first that there exists a universal constant $c>0$ and a tournament $B$ on $2^{ck}$ vertices with largest transitive subtournaments of order $k$ and that is $P_{k}$-free. 
Such a tournament may be easily constructed
randomly by fixing $2^{ck}$ vertices and choosing the direction of each edge independently at random with probability $\frac{1}{2}$ (standard probabilistic
argument shows that most of tournaments constructed according to this procedure satisfy the condition regarding sizes of their transitive subsets and $P_{k}$-freeness).

Now we define the following infinite family $\mathcal{F}$ of tournaments:
\begin{itemize}
\item $F_{0}$ is a one-vertex tournament,
\item $F_{i+1}=B \times F_{i}$ for $i=0,1,...$.
\end{itemize}

\begin{lemma}
Each tournament $F_{i} \in \mathcal{F}$ is $P_{k}$-free.
\end{lemma}

\begin{proof}
The proof is by induction on $i$. Induction base is trivial. Now let us assume that all $F_{i}$s for $i \leq i_{0}$ are $P_{k}$-free and let us
take tournament $F_{i_{0}+1}$. Denote the copies of $F_{i_{0}}$ that build $F_{i_{0}+1}$ as: $T_{1},...,T_{|S|}$.
Assume by contradiction that $P$ is a subtournament of $F_{i_{0}+1}$ that is isomorphic to $P_{k}$.
Notice first that $|V(P) \cap V(T_{j})| < k$ for $j=1,...,|S|$. Indeed, that follows from the fact that clearly every $T_{j}$ is $P_{k}$-free.
Now observe that if $|V(P) \cap V(T_{j})| > 0$ then in fact $|V(P) \cap V(T_{j})| = 1$. Otherwise, by the definition of $\mathcal{F}$ and from the
previous observation we would conclude that $V(P) \cap V(T_{j})$ is a nontrivial homogeneous subset of $V(P)$ but this contradicts the fact 
that $P$ is prime. But then we conclude that $P$ is a subtournament of $B$ which obviously contradicts the definition of $B$.
That completes the proof.
\end{proof}

Now notice that the size of $F_{i+1}$ is exactly $|B|$ times the size of $F_{i}$ and the size of the largest transitive subtournament of $F_{i+1}$
is exactly $tr(B)$ times the size of the largest transitive subtournament of $F_{i}$ for $i=0,1,...$, where $tr(B)$ stands for the size of the largest transitive subset of $B$.
That immediately leads to the conclusion that the size of the largest transitive subtournament of $F_{i}$ is of order $|F_{i}|^{\frac{\log(tr(B))}{\log(|B|)}}$.
The last expression, by the definition of $B$, is of order $|F_{i}|^{\frac{c\log(k)}{k}}$. Therefore $\mathcal{F}$ is the family we were looking for.

\section{Conclusions}


One can easily notice that our methods can be extended for larger classes of forbidden tournaments, for instance tournaments with the ordering of vertices under which the graph of backward 
edges is a matching. It would be interesting to characterize all classes of tournaments for which presented method (or its minor modifications) works.
The approximation ratio of the proposed algorithm may be in practice much better. This is another interesting direction that could be explored.

\section*{Acknowledgements} 
I would like to thank Dr. Marcin Pilipczuk for useful conversation about this work and very helpful suggestions regarding the manuscript. 


%
%

%

\end{document}